\documentclass[12pt,reqno,final]{amsart}
\usepackage{amsmath,amstext,amssymb,amscd,mathrsfs,amsthm}

\usepackage{verbatim}
\usepackage{graphicx}
\usepackage{amsthm}
\usepackage[notcite,notref]{showkeys}

\newtheorem{theorem}{Theorem}[section]
\newtheorem{lemma}[theorem]{Lemma}

\newtheorem{proposition}[theorem]{Proposition}
\newtheorem{assumption}[theorem]{Assumption}

\theoremstyle{definition}

\newtheorem{remark}{Remark}[section]

\theoremstyle{remark}

\numberwithin{equation}{section}
\begin{document}

\newcommand{\sgn}{\operatorname{sgn}}

\def\a{\alpha}
\def\b{\beta}
\def\d{\delta}
\def\g{\gamma}
\def\l{\lambda}
\def\o{\omega}
\def\s{\sigma}
\def\t{\tau}
\def\th{\theta}
\def\r{\rho}
\def\D{\Delta}
\def\G{\Gamma}
\def\O{\Omega}
\def\e{\epsilon}
\def\p{\phi}
\def\P{\Phi}
\def\S{\Psi}
\def\E{\eta}
\def\m{\mu}

\newcommand{\into}{\hookrightarrow}

\def\grad{\nabla}
\def\bar{\overline}
\newcommand{\reals}{\mathbb{R}}
\newcommand{\naturals}{\mathbb{N}}
\newcommand{\ints}{\mathbb{Z}}
\newcommand{\complex}{\mathbb{C}}
\newcommand{\rationals}{\mathbb{Q}}
\newcommand{\dualprod}[1]{\left\langle #1 \right\rangle}
\newcommand{\innerprod}[1]{\big( #1\big)}
\newcommand{\norm}[1]{\left\|#1\right\|}
\newcommand{\abs}[1]{\left|#1\right|}
\newcommand{\plap}{\Delta_p}
\newcommand{\plnd}[1]{|\nabla{#1}|^{p-2}\nabla{#1}}
\newcommand{\mlnd}[1]{|\nabla{#1}|^{m-2}\nabla{#1}}
\newcommand{\pldf}[1]{\textrm{div}(|\nabla{#1}|^{p-2}\nabla{#1})}
\newcommand{\Top}[1]{-\frac{1}{\lambda}\D #1-\plap #1 - f(\frac{a+#1}{\lambda}) + \lambda #1}
\newcommand{\hoo}{\textrm{H}^1_0(\Omega)}
\newcommand{\hoodual}{\textrm{H}^{-1}(\O)}
\newcommand{\ho}{\textrm{H}^1_0}
\newcommand{\hoe}{\textrm{H}^{1-\epsilon}(\O)}
\newcommand{\woop}{\textrm{W}_0^{1,p}(\Omega)}
\newcommand{\woom}{\textrm{W}_0^{1,m}(\Omega)}
\newcommand{\wop}{\textrm{W}_0^{1,p}}
\newcommand{\woopdual}{\textrm{W}^{-1,p'}(\O)}
\newcommand{\ltwo}{L^2(\O)}
\newcommand{\Lq}{L^q(\O)}
\newcommand{\lsp}[1]{L^{#1}(\O)}
\newcommand{\util}{\tilde{u}}
\newcommand{\dom}{\mathcal{D}}
\newcommand{\A}{\mathscr{A}}
\newcommand{\B}{\mathscr{B}}
\newcommand{\QE}{\mathscr{E}}
\newcommand{\ran}{\mathcal{R}}
\newcommand{\op}[2]{-\Delta {#1}  -\plap {#2} - f({#1})}

%%%%%%%%%%%%%%%%%%%%%%

%\def\reals{{\bf R}} 
\def\R{{\bf R}}
\def\M{{\bf M}} 
\def\N{{\bf N}} 
\def\O{{\bf O}} 
\def\S{{\bf S}} 
\def\d{\displaystyle} 
\def\e{{\varepsilon}} 
\def\o{\overline} 
\def\wt{\widetilde} 
\def\wh{\widehat} 
\def\vp{\varphi} 
\def\p{\partial} 
\def\v#1{\mbox{\boldmath $#1$}} 

\author[M.  Rammaha]{Mohammad Rammaha}
\address{Department of
Mathematics \\ University of Nebraska-Lincoln \\
Lincoln, NE  68588-0130, USA} \email{mrammaha1@unl.edu}

\author[H. Takamura]{Hiroyuki Takamura}
\address{Department of Complex and Intelligent Systems \\ Faculty of Systems Information Science\\Future University Hakodate\\
116-2 Kamedanakano-cho, Hakodate, Hokkaido 041-8655, Japan} \email{takamura@fun.ac.jp}

\author[H. Uesaka]{Hiroshi Uesaka }
\address{Department of Mathematics \\ College of Science and Technology\\ Nihon University\\
Chiyodaku Kanda Surugadai 1-8, Tokyo, 101-8308, Japan.} \email{uesaka@math.cst.nihon-u.ac.jp}

\author[K. Wakasa]{Kyouhei Wakasa}
\address{The 2nd year of the doctor course,
Department of Mathematics, Hokkaido University, Sapporo, 060-0810, Japan.}
\email{wakasa@math.sci.hokudai.ac.jp}

\title[Wave Equations ]{Blow-Up of Positive Solutions to  Wave Equations in High Space Dimensions}

%\date{\today}
\thanks{The second author is partially supported by Grant-in-Aid for Science Research (C) (No. 40241781), JSPS}

\begin{abstract} This paper is concerned with the Cauchy problem for the semilinear wave equation: 
\begin{align*}
u_{tt}-\Delta u=F(u) \quad \mbox{in}\quad\reals^n\times[0,\infty),
\end{align*}
where the space dimension $n\geq 2$,  $F(u)=|u|^p$ or $F(u)=|u|^{p-1}u$ with $p>1$. Here,  the Cauchy data are non-zero and  non-compactly supported. Our results on the blow-up of positive radial solutions (not necessarily radial in low dimensions $n=2, 3$) generalize and extend the results of Takamura \cite{Tak1} and Takamura, Uesaka and Wakasa \cite{Tak3}. The main technical difficulty in the paper lies in obtaining the lower bounds for the free solution when both initial position and initial velocity are non-identically zero in even space dimensions. 

\end{abstract}

\maketitle

\section {Introduction }\label{S1}

We consider the following Cauchy problem:
\begin{align}
\label{1.1}
\begin{cases}
u_{tt}-\Delta u=F(u) \quad \mbox{in}\quad\reals^n\times[0,\infty),\\
u(x,0)=f(x),\ u_t(x,0)=g(x)\quad \text{in}\quad\reals^n,
\end{cases}
\end{align}
where $n\ge2$, $u=u(x,t)$ is a scalar unknown function of space-time variables, and
$F(u)=|u|^p$ or $F(u)=|u|^{p-1}u$, with $p>1$. The scenario in one space-dimension is  fairly simple and one can always find general conditions on the initial data to prove the blow up of classical solution. Thus, we only consider problem (\ref{1.1})  in high  space dimensions, $n\geq 2$.

For the case of compactly supported initial data $\{f, \, g\}$ and when $F(u)=|u|^p$, we recall  Strauss' conjecture. Namely, there exists a critical number $p_0(n)$ such that (\ref{1.1})  has a global in time solution if the initial data are {\em sufficiently small} and $p>p_0(n)$; and (\ref{1.1}) has no global solutions if $1<p\leq p_0(n)$ and the initial data are positive in some sense. It was conjectured that $p_0(n)$ is the positive  root of the equation $(n-1)p^2-(n+1) p-2=0$. That is,
$$p_0(n)=\frac{1}{2(n-1)}\Big[n+1+\sqrt{n^2+10n-7}\Big].$$ 
We note here that $p_0(n)$ comes from the integrability of a certain weight function in the iteration argument for (\ref{1.1}).

The conjecture was first verified by John \cite{Jh1} for $n=3$, but not for $p=p_0(3)$. Glassey \cite{Gl-1,Gl-2}  verified the conjecture for $n=2$, but not for $p=p_0(2)$. The critical exponents $p=p_0(2)$ and $p=p_0(3)$ were proven by Schaeffer \cite{Sch}. In high space dimensions, $n\geq 4$, the subcritical case $1<p< p_0(n)$ was handled by Sideris \cite{sid1}, and later by Rammaha \cite{R7} who provided a simplified proof.   The super critical case $p>p_0(n)$ was proven by Georgiev, Lindblad, and Sogge \cite{GLS}. Finally, the critical case $p=p_0(n)$, $n\geq 4$ was handled by Yordanov and Zhang \cite{YZ}, and independently by Zhou \cite{Zhou}. Thus, Strauss' conjecture has been completely resolved and
all of  the cited results above on Strauss' conjecture are summarized in the following table:
\begin{center}
\begin{tabular}{|c||c|c|c|}
\hline
 & $1<p<p_0(n)$ & $p=p_0(n)$ & $p>p_0(n)$ \\
\hline
\hline
$n=2$ & \cite{Gl-2} & \cite{Sch} & \cite{Gl-1} \\
\hline
$n=3$ & \cite{Jh1} & \cite{Sch} & \cite{Jh1} \\
\hline
$n\ge4$ & \cite{sid1} & \cite{YZ}, \cite{Zhou} (independently) & \cite{GLS} \\
\hline
\end{tabular}
\end{center}

However, the scenario is somewhat different when the initial data are not compactly supported and decaying slowly at infinity. In fact,  problem (\ref{1.1}) may have no global solution even for the supercritical case ($p>1$ is arbitrarily large). Indeed, the pioneering results on non-compactly supported initial  data  by Asakura \cite{As} strongly suggests the validity of the following statement:
\begin{equation}\label{conj-slow}
\fbox{
\parbox{0.8\textwidth} {\centering
 There exists a critical decay exponent   $\kappa_0 >0$ such that  (\ref{1.1}) has a global solution, provided $\kappa \geq\kappa_0$, $p>p_0(n)$ and the initial data are sufficiently small, yet (\ref{1.1}) has no global solutions, provided $0<\kappa <\kappa_0$, $p>1$, and the initial data are positive in some sense.
} }
\end{equation}
It is remarkable that (see for instance \cite{As})  
the critical decay exponent  $\kappa_0$ is independent of the space dimension $n$ and it is given by:
\begin{align}\label{kappa_0}
\kappa_0 = \frac{2}{p-1},\,\,\,\ p>1.
\end{align}  
As shown in \cite{As}  and later by
Takamura, Uesaka and Wakasa \cite{Tak3} that there exists  a constant $L>0$ such that (\ref{1.1}) has no global solution if the initial data $\{f, \, g\}$ satisfy:
\begin{align}
\label{condi:cblow-up1}
f(x)\equiv0\quad\mbox{and}\quad g(x)\ge\frac{\phi(|x|)}{(1+|x|)^{1+\kappa}},
\end{align}
or
\begin{align}
\label{condi:cblow-up2}
f(x)>0,\ \Delta f(x)+F(f(x))\ge\frac{\phi(|x|)}{(1+|x|)^{2+\kappa}}
\quad\mbox{and}\quad g(x)\equiv0,
\end{align}
for all $|x|\ge L$, with
\begin{align}
\label{condi:phi1}
0<\kappa<\kappa_0\quad\mbox{and}\quad\phi(x)\equiv\mbox{positive const.},
\end{align}
or
\begin{align}
\label{condi:phi2}
\kappa=\kappa_0,\quad\phi\ \mbox{is positive, monotonously increasing and}
\lim_{|x|\rightarrow\infty}\phi(|x|)=\infty.
\end{align}
On the other hand (see for instance the results in  \cite{Kub1}), (\ref{1.1}) has a global solution provided
\begin{align}
\label{condi:cglobal}
(1+|x|)^{1+\kappa}\left(\frac{|f(x)|}{1+|x|}+\sum_{0<|\alpha|\le[n/2]+2}|\nabla_x^\alpha f(x)|
+\sum_{|\beta|\le[n/2]+1}|\nabla_x^\beta g(x)|\right)
\end{align}
is sufficiently small, $\kappa\ge\kappa_0$ and $p>p_0(n)$, where $n=2,3$. 
In high odd space dimensions  $n=2m+1$, $m\ge2$, Kubo's results \cite{Kubo96} shows that the radially symmetric version of (\ref{1.1}) has a global solution, provided
\begin{align}
\label{condi:cglobal_rad}
\sum_{j=0}^2|f^{(j)}(r)|\langle r\rangle^{\kappa+j}
+\sum_{j=0}^1|g^{(j)}(r)|\langle r\rangle^{1+\kappa+j} 
\end{align} 
is sufficiently small,  where $\langle r\rangle=\sqrt{1+r^2}$ and $r=|x|$. 
A similar result was obtained by Kubo and Kubota \cite{KK} in the  case of even space dimensions $n=2m$, $m\ge2$, but under a more stringent condition than (\ref{condi:cglobal_rad}) near $r=0$.
We note that the similar result for the equation with the potential has obtained by Karageorgis \cite{Kara05}.

When $n=3$, Asakura \cite{As} was the first to prove the nonexistence result 
under the validity of  (\ref{condi:cblow-up1}) or  (\ref{condi:cblow-up2}). In addition,  Asakura \cite{As} resolved the existence part under assumption  (\ref{condi:cglobal}). 
The critical case ($\kappa=\kappa_0$) and $n=3$  was handled  by Kubota \cite{Kub1} with  the assumption (\ref{condi:cglobal}), and also independently by Tsutaya \cite{Ts1}. 
For $n=2$, the nonexistence part with (\ref{condi:cblow-up1}), (\ref{condi:cblow-up2}) was verified by Agemi and Takamura \cite{Ag-Tak}, and the existence part was verified by 
Kubota \cite{Kub1}, and both parts by Tsutaya \cite{Ts2,Ts3}. 
We note here that the all of existence results mentioned above were proven without the presence of  
$(1+|x|)^{-1}$ in the first term of (\ref{condi:cglobal}). 
However, Kubota and Mochizuki \cite{KubotaMochizuki}, in their work to prove the existence of the scattering operator, were the first to introduce (\ref{condi:cglobal}) when $n=2$. 
In higher space dimensions and for radial solutions, 
the existence part with (\ref{condi:cglobal_rad}) was handled by Kubo \cite{Kubo96}, Kubo and Kubota \cite{Ku2, KK}, 
and the nonexistence part with (\ref{condi:cblow-up1}) and (\ref{condi:cblow-up2}) was verified by Takamura \cite{Tak1}. Another relevant 
 nonexistence result in high dimensions $n\ge2$  is due to  Kurokawa and Takamura \cite{KuroTak}.
In view of (\ref{condi:cglobal_rad}), we note that the final form of (\ref{condi:cglobal}) will be 
\begin{align}
\label{condi:cglobal_final}
\sum_{|\alpha|\le[n/2]+2}(1+|x|)^{\kappa+|\alpha|}|\nabla_x^\alpha f(x)|
+\sum_{|\beta|\le[n/2]+1}(1+|x|)^{\kappa+1+|\beta|}|\nabla_x^\beta g(x)|.
\end{align}

All of the cited results on non-compactly supported initial data are summarized in the following tables.
\begin{center}
\begin{tabular}{|c||c|c|c|}
\hline
Global existence & $\kappa=\kappa_0$ & $\kappa>\kappa_0$ \\
\hline
\hline
$n=2$ & \cite{Kub1}, \cite{Ts2}  independently & \cite{Kub1}, \cite{Ts3} independently\\
\hline
$n=3$ & \cite{Kub1}, \cite{Ts1}  independently & \cite{As} \\
\hline
$n\ge4$ & \cite{Kubo96} & \cite{Ku2} and \cite{KK} \\
\hline
\end{tabular}
\end{center}
\begin{center}
\begin{tabular}{|c||c|c|}
\hline
Blow-up & (\ref{condi:cblow-up1}) & (\ref{condi:cblow-up2})\\
\hline
\hline
(\ref{condi:phi1}) & \cite{Ag-Tak}, \cite{Ts3} independently \mbox{for} $n=2$ & \cite{Tak2}\\ 
& \cite{As} \mbox{for} $n=3$ &\\
& \cite{Tak1} \mbox{for} $n\ge4$ & \\
\hline
(\ref{condi:phi2}) & \cite{KuroTak} & \cite{Tak3}\\
\hline
\end{tabular}
\end{center}
\smallskip

Thus far, all of the cited nonexistence results above were proven with zero initial position, except for Takamura, Uesaka and Wakasa \cite{Tak2,Tak3} who
proved a  nonexistence result under the nonzero initial position with the assumption (\ref{condi:cblow-up2}) by differentiating  (\ref{1.1}) with respect to time. 

In this paper we prove a  blow-up result with sharp decay
for $f\not\equiv0$ and $g\not\equiv0$. 
The main goal of this work is to obtain the required point-wise lower bounds for the {\em free} solution of the  wave equations 
by making full use of the formulas by Rammaha \cite{R3,R7} in high dimensions. 
In low space dimensions, one can obtain such lower bounds solutions as in Caffarelli and Friedman \cite{CF}. 
However, it is highly nontrivial to obtain the mentioned lower bounds for the free solution when both initial data are non-zero, particularly in high even dimensions.  We overcome the main technical difficulty in  high even dimensions by introducing a special  change of variables given in (\ref{change_vari}).

%%%%%%%%%%%%%%%%%%%%%%%%%%%%%%%%%%%%%%%%%%%%%%%%%%%%%%%%%%%%%%%%%%%%%%%%%%%%%%%%%%%%
\section{Main results} \label{sec1}
In high space dimensions $n\geq 4$, we restrict our analysis to radial solutions. More precisely,
we consider the following radially symmetric version of (\ref{1.1}):
\begin{align}\label{2.1}
\begin{cases}
u_{tt}-\d\frac{n-1}{r}u_r-u_{rr}=F(u),  \mbox{   in   }\ (0,\infty)\times[0,\infty),\vspace{.1in}\\
u(r,0)=f(r),\ u_t(r,0)=g(r), \mbox{   in  }\  (0,\infty).
\end{cases}
\end{align}
Henceforth, our assumptions (see Assumption \ref{ass} below) in high dimensions $n\geq 4$ are in reference of the Cauchy problem (\ref{2.1}).

In order to  state our main results, we begin with the assumptions on the initial data and the parameters.

\begin{assumption} \label{ass}
\hfill
\begin{itemize}
\item \textbf{The nonlinearity:}  $F\in C^1(\reals)$ 
satisfying
\begin{align}
\label{asm_F}
F(s)\ge As^{p}, \,\,\, \mbox{ for  }\ s\ge0,
\end{align}
where $p>1$ and $A>0$.
\vspace{.1in}

\item \textbf{Low space dimensions, $n=2, 3$:} There exists a constant $R>0$ such that the initial data satisfying:  $f \in C^3(\reals^n)$ 
and $g \in C^2(\reals^n)$ and such that
\begin{align}
\label{asm_low}
\begin{cases}
f(x)>0 \ \mbox{ for }\ |x|\ge R, \vspace{.1in}\\
\d \frac{f(x)}{1+|x|}-|\nabla f(x)|+g(x)\ge \frac{C_0}{(1+|x|)^{1+\kappa}}, 
\ \mbox{ for }\ |x|\ge R,
\end{cases}
\end{align} 
for some positive constants $C_0$ and $\kappa$. 
\vspace{.1in}

\item \textbf{High space dimensions, $n\geq 4$:}  There exists a  constant $R>0$ such that 
$f\in C^2(0,\infty)$ and $g\in C^1(0,\infty)$ satisfying 
\begin{align}\label{asm_odd_1}
\begin{cases}
\d f(r)\ge \frac{C_1}{(1+r)^{\kappa}},\ g(r)>0\\
\d-C_{1,m}\frac{f(r)}{r}+g(r)>0,\ \mbox{   for   }\ r\ge R, 
\end{cases}
\end{align}
or
\begin{align}\label{asm_odd_2}
\begin{cases}
\d f(r),g(r)>0\\
\d-C_{1,m}\frac{f(r)}{r}+g(r)\ge\frac{C_2}{(1+r)^{1+\kappa}},\ \mbox{   for   }\ r\ge R, 
\end{cases}
\end{align}
if $n=2m+1$, 
\begin{align}\label{asm_even}
\begin{cases}
\d f(r),g(r)>0\\
\d -C_{2,m}\frac{f(r)}{r}-|f'(r)|+\frac{1}{2} g(r)
\ge\frac{C_3}{(1+r)^{1+\kappa}}, \ \mbox{   for   }\ r\ge R,
\end{cases}
\end{align}
if $n=2m$, where $m=2,3, \cdots$, $C_1$, $C_2$ and $C_3$ are positive constants, and the  constants  
$C_{1,m}$ and $C_{2,m}$ are given by 
\[
C_{1,m}=m(m-1),\quad 
C_{2,m}=m-\frac{3}{8}+\frac{5\zeta_m(m-1)^2}{3},
\]
where  $\zeta_m>0$ is as determined in Lemma \ref{lem_4.1}. 
\vspace{.1in}

\item \textbf{Parameters:} $0<\kappa<\kappa_0$, where  $\d \kappa_0 = \frac{2}{p-1}$.
\end{itemize}
\end{assumption}

Our first  result is on the finite-time blow up  of classical solutions in low dimensions, without imposing radial symmetry.
\begin{theorem}
\label{th_2.1}
Assume the validity of Assumption \ref{ass} with  $n=2$ or $n=3$, and $u$ is a solution of (\ref{1.1}).
Then $u$ cannot exist globally in time.
\end{theorem}

Our second result addresses the  finite-time blow up of radial solutions to the Cauchy problem (\ref{2.1}).   \begin{theorem}
\label{th_2.2}
Assume the validity of Assumption \ref{ass} with  $n\geq  4$, and $u$ is a solution of (\ref{2.1}).
Then $u$ cannot exist globally in time.
\end{theorem}

%\begin{remark}\label{rmk1}
%By the statements ``$u$ is a solution"  in theorems \ref{th_2.1} and  \ref{th_2.2} we mean $u $ is a classical $C^2$-solution   if $F\in %C^2(\reals)$, or  $u$ is a $C^1$-solution of the associated 
%integral equation for (\ref{1.1}) or (\ref{2.1}) if $F\in C^1(\reals)$. 
%\end{remark}

\begin{remark}\label{rmk2}
Let us note here that our assumption on the initial data in  
(\ref{asm_odd_1}), (\ref{asm_odd_2}) and (\ref{asm_even}) are fairly reasonable  in view of the slowly decaying initial data (see for instance (\ref{condi:cglobal_rad}) or remark 2.1 in \cite{Tak2}). 
In fact, there is a large family of the slowly decaying initial data that satisfies  the general conditions in Assumption \ref{ass}.
\end{remark}

The paper is organized as follows. In the next section, we illustrate our iteration schemes, which are sufficient to prove Theorem \ref{th_2.1} and Theorem \ref{th_2.2}. 
Section \ref{sec3} is devoted to the treatment of high odd dimensions. In Section \ref{sec4}, we derive the required lower bound in high even dimensions, which is the more technical part of the paper.  Finally, Section \ref{sec5} gives a brief treatment of the low dimensions $n=2, 3$.
%%%%%%%%%%%%%%%%%%%%%%%%%%%%%%%
%%%%%%%%%%%%%%%%%%%%%%%%%%%%%%%
\section{Iteration Scheme} \label{sec2}
In this section, we introduce our iteration scheme that allows us to prove the Theorem \ref{th_2.1} and Theorem \ref{th_2.2}, following the  well-known arguments in  \cite{Jh1} or \cite{Tak1}.  Throughout the paper, we define $\delta$ (which depends of the space dimensions $n$) by:
\begin{equation}\label{def-delta}
\delta:= \max\Big\{\frac{2}{\eta_m},\,\, \frac{2}{\zeta_m} \Big\},
\end{equation}
where $\eta_m,\,\,\zeta_m >0 $ are given below in Lemma \ref{lem_P} and Lemma \ref{lem_4.1}; respectively.
\begin{lemma}\label{lem_frame_high}
Let $u$ be a solution of (\ref{2.1}) where $n=2m+1$ or $n=2m$, $m=2,3,4\cdots$. 
Then, with the validity of  Assumption \ref{ass}, we have: 
\begin{align}
\label{frame_high}
u(r,t)\ge \frac{C_4t}{(1+r+t)^{1+\kappa}}+\frac{1}{8r^m}\int_{0}^{t}d\tau
\int_{r-t+\tau}^{r+t-\tau}\lambda^mF(u(\lambda,\tau))d\lambda,
\end{align}
for all $(r,t)\in \Sigma_1$, where 
\begin{align}
\label{blowup_set_1}
\Sigma_1:=\left\{(r,t)\in(0,\infty)^2\ :\ r-t\ge
\max\left\{R,\,\,\delta t\right\}>0\right\}
\end{align}
and  $C_4$  is a  positive constant.
\end{lemma}
\begin{lemma}\label{lem_frame_low}
Let $u$ be a solution of (\ref{1.1}) with $n=2$ or $n=3$, and  Assumption \ref{ass} is valid. 
Then we have:
\begin{align}
\label{frame_low}
u(x,t)\ge \frac{C_5t}{(1+|x|+t)^{1+\kappa}}+\int_{0}^{t}R(F(u(\cdot,\tau))|x,t-\tau)d\tau,
\end{align}
for all $(x,t)\in \Sigma_2$, where $C_5$ is a positive constant,
\begin{align}
\label{riemman}
R(\phi|x,t):=
\left\{
\begin{array}{ll}
\d\frac{t}{4\pi}\int_{|\omega|=1}\phi(x+t\omega)dS_{\omega}
&\mbox{for}\ n=3, \vspace{.1in}\\
\d\frac{1}{2\pi}\int_{0}^{t}\frac{\rho d\rho}{\sqrt{t^2-\rho^2}}
\int_{|\omega|=1}\phi(x+\rho\omega)dS_{\omega}
&\mbox{for}\ n=2,
\end{array}
\right.
\end{align}
and
\begin{align}
\label{blowup_set_2}
\Sigma_2:=\left\{(x,t)\in\R^n\times(0,\infty)\ :\ |x|-t\ge
\max\{R,t-1\}\right\}.
\end{align}
\end{lemma}
As we mentioned earlier, by appealing to iteration arguments in \cite{Jh1} or \cite{Tak1} along with  Lemma \ref{lem_frame_high} and Lemma \ref{lem_frame_low}, 
one can prove the Theorems \ref{th_2.1} and Theorems \ref{th_2.2}. 
\smallskip

The proofs of the above lemmas are provided below. First, let $u^0$ denotes the free
solution of the wave equation. More precisely, $u^0$ is the solution of the Cauchy problem:
\begin{align}\label{u^0_low}
\begin{cases}
u_{tt}^0-\Delta u^0=0, \mbox{    in   }\reals^n\times[0,\infty),\vspace{.1in}\\
u^0(x,0)=f(x),\ u_t^0(x,0)=g(x), \mbox{    in   }\reals^n,
\end{cases}
\end{align}
if $n=2$ or $n=3$ (no radial symmetry is assumed), and 
for $n\ge 4$, $u^0$ is the solution of  the following radially symmetric version of (\ref{u^0_low}):
\begin{align}
\label{u^0_high}
\begin{cases}
u_{tt}^0-\d\frac{n-1}{r}u_r^0-u_{rr}^0=0, \mbox{    in   } (0,\infty)\times[0,\infty),\vspace{.1in} \\
u^0(r,0)=f(r),\ u_t^0(r,0)=g(r), \mbox{    in   } (0,\infty).
\end{cases}
\end{align}
Then, we have the following results.
\begin{proposition}\label{prop_frame_high}
Let $u^0$ be the solution of (\ref{u^0_high}) with $n=2m+1$ or $n=2m$, $m=2,3,4,\cdots$. 
With the validity of  Assumption \ref{ass}, then 
$u^0$ satisfies:
\begin{align}\label{u^0_est_pre_odd}
u^0(r,t) &\ge
\frac{1}{2r^m}\Big\{f(r+t)(r+t)^m+f(r-t)(r-t)^m\Big\} \notag\\
&\quad+\frac{1}{4r^m}\int_{r-t}^{r+t}\lambda^m\left(-C_{1,m}\frac{f(\lambda)}{\lambda}
+g(\lambda)\right)d\lambda
\end{align}
for all $(r,t)\in \Sigma_1$, if $n=2m+1$, and
\begin{align}\label{u^0_est_pre_even}
u^0(r,t)
&\d\ge \frac{1}{\pi r^{m-1}}\int_0^t\frac{\rho d\eta}{\sqrt{t^2-\rho^2}}
\int_{r-\rho}^{r+\rho}\left\{-2C_{2,m}\frac{f(\lambda)}{\lambda}
-2|f'(\lambda)|+g(\lambda)\right\}\notag\\
&\quad\times\frac{\lambda^md\lambda}{\sqrt{\lambda^2-(r-\rho)^2} 
\sqrt{(r+\rho)^2-\lambda^2}}
\end{align}
for all $(r,t)\in \Sigma_1$, if $n=2m$.
\end{proposition}
\begin{proposition}\label{prop_frame_low}
Let $n=2$ or $n=3$, $u^0$ be the solution of (\ref{u^0_low}) and Assumption \ref{ass} is valid. 
Then, $u^0$ satisfies: 
\begin{align}\label{u^0_est_pre_3}
u^0(x,t) \ge
\frac{t}{4\pi}\int_{|\omega|=1}\left\{
\frac{f(x+t\omega)}{1+|x+t\omega|}-|\nabla f(x+t\omega)|
+g(x+t\omega)\right\}dS_{\omega}
\end{align}
for all $(x,t)\in \Sigma_2$, if $n=3$,
\begin{align}\label{u^0_est_pre_2}
u^0(x,t) &\ge
\frac{1}{2\pi}\int_{0}^{t}\frac{\rho d\rho}{\sqrt{t^2-\rho^2}}\notag\\
&\quad\times\int_{|\omega|=1}\left\{\frac{f(x+\rho\omega)}{1+|x+\rho\omega|}-|\nabla f(x+\rho\omega)|
+g(x+\rho\omega)\right\}dS_{\omega}
\end{align}
for all $(x,t)\in \Sigma_2$, if $n=2$.
\end{proposition}

It is important to  note here that the proofs of Lemma \ref{lem_frame_high} and Lemma \ref{lem_frame_low} 
follow from Propositions \ref{prop_frame_high} and \ref{prop_frame_low},  and by appealing to the proofs of Lemma 2.1 in \cite{KuroTak} or Lemma 2.6 and lemma 2.9 in \cite{Tak1}. 
For this very reason,  the remaining parts of the paper are devoted only to the proofs of Proposition \ref{prop_frame_high} and Proposition \ref{prop_frame_low}. 

%%%%%%%%%%%%%%%%%%%%%%%%%%%%%%%
%%%%%%%%%%%%%%%%%%%%%%%%%%%%%%%
\section{High Odd Dimensions: $n=5,7,9,\cdots$}\label{sec3}
{\bf Proof of proposition \ref{prop_frame_high} in \boldmath{$n=2m+1,\ m=2,3,4,\cdots$.}} 
According to formula (6a) in \cite{R3}, we have 
\begin{align}
\label{u^0_odd_1}
\begin{array}{ll}
u^0(r,t)=
&\d\frac{\p}{\p t}\left\{\frac{1}{2r^m}\int_{|r-t|}^{r+t}\lambda^mf(\lambda)P_{m-1}
\left(\Theta(\lambda,r,t)\right)d\lambda\right\}\\
&\d+\frac{1}{2r^m}\int_{|r-t|}^{r+t}\lambda^mg(\lambda)P_{m-1}
\left(\Theta(\lambda,r,t)\right)d\lambda,
\end{array}
\end{align}
where $P_k$ denotes Legendre polynomials of degree $k$ defined by
\begin{align}
\label{P}
P_k(z):=\frac{1}{2^kk!}\frac{d^k}{dz^k}(z^2-1)^k,
\end{align} 
and $\Theta=\Theta(\lambda,r,t)$ is given by 
\begin{align}
\label{Xi}
\Theta(\lambda,r,t)=\frac{\lambda^2+r^2-t^2}{2r\lambda}.
\end{align}
The following auxiliary lemma will be needed in the  derivation of  the required estimate in this case. 

\begin{lemma}
\label{lem_P}
For $m=2,3,4,\cdots$, there exists a positive constant $\eta_m$, depending only on 
$m$, such that 
\begin{align}\label{est_P}
P_{m-1}(z)\ge\frac{1}{2}\quad\mbox{and}\quad0<P_{m-1}'(z)\le\frac{1}{2}m(m-1),
\quad\mbox{for}\quad
\frac{1}{1+\eta_m}\le z\le 1.
\end{align}
\end{lemma}
\par\noindent
\begin{proof}
Let us first consider the case of $m=2$. Then, we easily obtain (\ref{est_P}) 
by putting $\eta_m=1$, since $P_{1}(z)=z$. 
Now, suppose that $m\ge3$. Then, by direct computations, we have the following properties of $P_k$:
\begin{align}
\label{val_P_1}
P_{m-1}'(1)=\frac{1}{2} m(m-1)>0,
\end{align}
and
\begin{align}
\label{val_P_2}
P_{m-1}''(1)=\frac{1}{4} (m-1)(m-2){m+1\atopwithdelims()m-1}>0.
\end{align}
Since  $P_1(z)=z$ and $P_m(1)=1$, then it follows from (\ref{val_P_1}), (\ref{val_P_2}) and 
the continuity of $P_{m-1}(z)$, $P'_{m-1}(z)$ and $P''_{m-1}(z)$ that there exists a $\eta_m>0$ such that (\ref{est_P}) is valid.
\end{proof}
To use the lemma \ref{lem_P} with $\Theta$ which is a variable of $P_{m-1}$ or $P'_{m-1}$ in (\ref{u^0_odd_1}), 
we need following lemma. 
\begin{lemma}
\label{lem_Xi}
Let $\Theta$ be the function defined by (\ref{Xi}). Then, $\Theta$ satisfies 
\begin{align}\label{est_Xi}
\Theta(\lambda,r,t)\ge \frac{\delta}{\delta+2}
\quad\mbox{for}\quad r-t\le \lambda \le r+t
\end{align}
provided $(r,t)\in \Sigma_1$. 
\end{lemma}
\begin{proof}
Its easy to see that 
\[
\Theta(\lambda,r,t)\ge \frac{(r-t)^2+r^2-t^2}{2r(r+t)}=\frac{r-t}{r+t}\ge \frac{\delta}{\delta+2}
\]
for $r-t\le \lambda \le r+t$ and $(r,t)\in \Sigma_1$.
\end{proof}

Let us first note that  (\ref{u^0_odd_1}) yields:
\begin{align}
\label{u^0_odd_2}
\begin{array}{ll}
u^0(r,t)=
&\d\frac{1}{2r^m}\left\{f(r+t)(r+t)^m+f(r-t)(r-t)^m\right\}\\
&\d+\frac{1}{2r^m}\int_{r-t}^{r+t}\lambda^mf(\lambda)P_{m-1}'
\left(\Theta(\lambda,r,t)\right)
\left(-\frac{t}{r\lambda}\right)d\lambda\\
&\d+\frac{1}{2r^m}\int_{r-t}^{r+t}\lambda^mg(\lambda)P_{m-1}
\left(\Theta(\lambda,r,t)\right)d\lambda.
\end{array}
\end{align}
Thanks to  lemma \ref{lem_P}, lemma \ref{lem_Xi} and the assumption (\ref{asm_odd_1}) or (\ref{asm_odd_2}), 
then  (\ref{u^0_est_pre_odd}) holds, for all $(r,t)\in \Sigma_1$. 
Hence, the proof of the proposition \ref{prop_frame_high} in odd space dimension $n=2m+1,\ m=2,3,4,\cdots$ is complete. 
\hfill$\Box$

\smallskip
{\bf Completion of the  Proof of Lemma \ref{lem_frame_high}.}  
Let us note here that  the first term in (\ref{frame_high}) of lemma \ref{lem_frame_high} is obtained as follows. Thanks to (\ref{asm_odd_1}), then 
(\ref{u^0_est_pre_odd}) yields
\begin{align}
\label{u^0_odd_est_1}
\begin{array}{lll}
u^0(r,t)&\ge
\d\frac{C_1}{2(1+r+t)^{\kappa}}\left\{1+\left(\frac{r-t}{r}\right)^m\right\}\\
&\ge\d\frac{C_1}{2}\left\{1+\left(\frac{2}{3}\right)^m\right\}\frac{t}{(1+r+t)^{1+\kappa}}
\end{array}
\end{align}
for all $(r,t)\in \Sigma_1$. Since $\eta_m\le 1$ for all $m=2,3,4,\cdots$ 
such that $r\ge3t$ holds. Hence, the first term in (\ref{frame_high}) valid for $\d C_4=C_12^{-1}\left\{1+(3/2)^m\right\}$. 
Next, we shall show by using the assumption (\ref{asm_odd_2}). Then, (\ref{u^0_est_pre_odd})  yields 
\begin{align}
\label{u^0_odd_est_2}
\begin{array}{lll}
u^0(r,t)&\ge\d\frac{C_2}{4r^m}\int_{r-t}^{r+t}\lambda^m(1+\lambda)^{-\kappa-1}d\lambda\\
&\ge\d\frac{C_2t}{4(1+r+t)^{\kappa+1}}.
\end{array}
\end{align}
for all $(r,t)\in \Sigma_1$.  
Hence, the first term in (\ref{frame_high}) valid for $C_4=C_2/4$.
\hfill$\Box$

%%%%%%%%%%%%%%%%%%%%%%%%%%%%%%%%%%%%%%%%%%%%%%%%%%%%%%%%%%%%%%%%%%%%%%%%%%%%%%%%%%%%
\section{High Even Dimensions: $n=4,6,8\cdots$} \label{sec4}
{\bf Proof of proposition \ref{prop_frame_high} in \boldmath{$n=2m,\ m=2,3,4,\cdots$.}} 
According to  formula (6b) in \cite{R3}, we have
\begin{align} 
\label{u_0_even} 
u^0(r,t)= \frac{\p}{\p t}\frac{2}{\pi r^{m-1}}I(r,t,u^0(\cdot,0))
+\frac{2}{\pi r^{m-1}}I(r,t,u_t^0(\cdot,0)),
\end{align} 
where,
\begin{align} 
\label{I} 
I(r,t,\psi(\cdot))
=\int_0^t\frac{\rho d\rho}{\sqrt{t^2-\rho^2}} 
\int_{|r-\rho|}^{r+\rho} 
\frac{\lambda^m\psi(\lambda) 
T_{m-1}\left(\Theta(\lambda,r,\rho)\right)d\lambda} 
{\sqrt{\lambda^2-(r-\rho)^2} 
\sqrt{(r+\rho)^2-\lambda^2}}
\end{align} 
and as usual, in (\ref{I}) $T_k$ denotes Tschebyscheff polynomials of degree $k$ defined by 
\begin{align} 
\label{T} 
T_k(z):=\frac{(-1)^k}{(2k-1)!!}(1-z^2)^{1/2}\frac{d^k}{dz^k}(1-z^2)^{k-(1/2)} 
\end{align} 

The following auxiliary lemma will be needed. 
\begin{lemma}
\label{lem_4.1}
For $m=2,3,4,\cdots$, there exists a positive constant $\zeta_m$, depending only on 
$m$, such that 
\begin{align}
\label{T_1}
\begin{cases}
\frac{1}{2}\le T_{m-1}(z)\le 1, \vspace{.1in}\\
0<T_{m-1}'(z)\le(m-1)^2,
\end{cases}
\end{align}
for all $\d \frac{1}{1+\zeta_m}\le z \le 1$.
\end{lemma}

\begin{proof}
Let us first consider the case of $m=2$. Then, we easily obtain (\ref{T_1}) since $T_{1}(z)=z$
and we may take  $\zeta_m=1$. 
Now, let $m\ge3$. Since $T_{m-1}(1)=1$, then the first assertion is trivial as long as $\zeta_m>0$ is sufficiently small.  For $m\geq 2$, we recall that  the Tchebysheff polynomial $T_{m-1}(z)$ satisfies the ODE:
\begin{align} 
\label{diff_T} 
(1-z^2)T_{m-1}''(z)-zT'_{m-1}(z)+(m-1)^2T_{m-1}(z)=0\quad \mbox{for}\quad  |z| \le 1. 
\end{align} 
Thus,  (\ref{diff_T}) yields
\begin{align} \label{T'1} 
\begin{cases}
T'_{m-1}(1)=(m-1)^2, \vspace{.1in}\\
T''_{m-1}(1)=\frac{1}{3}m(m-2)(m-1)^2,
\end{cases}
\end{align} 
for $m\geq 3$. Hence, the second assertion of the lemma follows from continuity $T'_{m-1}(z)$ and $T''_{m-1}(z)$.
\hfill$\Box$

\medskip

In order to obtain the desired lower bound in high even dimensions, we shall use the  following change variables in (\ref{I}):
For $(r,t)\in \Sigma_1$, we introduce: 
\begin{align}
\label{change_vari}
\begin{cases}
\d \xi=\frac{r+\rho-\lambda}{2\rho}\quad \mbox{in $\lambda$-integral}, \\ 
\d \rho=t\eta\quad \mbox{ in $\rho$-integral}.
\end{cases}
\end{align}
Then, with this change of variables then (\ref{I}) reduces to: 
\begin{align} 
\label{I_new}
I(& r,t,\psi(\cdot))=\notag \\
&\frac{t}{2}\int_{0}^{1}\frac{\eta d\eta}{\sqrt{1-\eta^2}}
\int_{0}^{1}\frac{K(r,t,\eta,\xi)\psi(r+t\eta-2t\eta\xi)}{\sqrt{\xi}\sqrt{1-\xi}}  T_{m-1}\left(\Theta(r,t,\eta,\xi)\right)d\xi.
\end{align}
In addition,
\begin{align}
\label{diff_I} 
\frac{\p}{\p t}& I(r,t,\psi(\cdot))
=\notag \\ 
&\frac{1}{2}\int_0^1\frac{\eta d\eta}{\sqrt{1-\eta^2}}
\int_{0}^{1}\biggl\{K(r,t,\eta,\xi)\psi(r+t\eta-2t\eta\xi)T_{m-1}(\Theta(r,t,\eta,\xi))\notag\\
&+t\frac{\p}{\p t}\left\{K(r,t,\eta,\xi)\psi(r+t\eta-2t\eta\xi)T_{m-1}(\Theta(r,t,\eta,\xi))\right\}\biggr\}\frac{d\xi}{\sqrt{\xi}\sqrt{1-\xi}},
\end{align} 
where,
\begin{align} 
\label{K} 
K(r,t,\eta,\xi)=\frac{(r+t\eta-2t\eta\xi)^m}{\sqrt{r+t\eta-t\eta\xi}
\sqrt{r-\xi t\eta}} 
\end{align} 
and
\begin{align}
\label{theta}
\Theta(r,t,\eta,\xi):=\Theta\left(r+t\eta-2t\eta\xi,r,t\eta\right)=\frac{(r+t\eta-2t\eta\xi)^2+r^2-t^2\eta^2}{2r(r+t\eta-2t\eta\xi)}.
\end{align}

The following proposition is crucial to the rest of the proof. 
\begin{proposition}
\label{prop_4.2}
Let $m=2,3,4\cdots$. Assume that $w\in C^1((0,\infty))$ and $w(y)>0$ for 
$y\ge R$, where $R$ is as given in (\ref{asm_even}). Then, for $0\le \xi,\eta \le 1$ and $(r,t)\in\Sigma_1$, we have
\begin{align} \label{diff_kernel}
\frac{\p}{\p t}& \left\{K(r,t,\eta,\xi)w(r+t\eta-2t\eta\xi)T_{m-1}(\Theta(r,t,\eta,\xi))\right\} \notag\\
& \ge-\left\{E_m\frac{w(r+t\eta-2t\eta\xi)}{r+t\eta-2t\eta\xi} +|w'(r+t\eta-2t\eta\xi)|\right\}
K(r,t,\eta,\xi),
\end{align}
where $E_m$ is defined by 
\[
E_m=m+\frac{1}{8}+\frac{5\zeta_m(m-1)^2}{3}.
\]
\end{proposition}
\par\noindent
\begin{proof}
By direct computation, we obtain: 
\[ 
\begin{array}{lll} 
& \d\frac{\p}{\p t}\Big\{K(r,t,\eta,\xi)w(r+t\eta-2t\eta\xi)T_{m-1}(\Theta(r,t,\eta,\xi))\Big\}&\\ 
&\d=K(r,t,\eta,\xi)\Big\{\eta \left\{I_1+I_2+I_3+I_4\right\}T_{m-1}(\Theta(r,t,\eta,\xi))+I_5\Big\}, 
\end{array} 
\] 
where,
\begin{align}\label{I_i}
I_1&=m(1-2\xi)\frac{w(r+t\eta-2t\eta\xi)}{r+t\eta-2t\eta\xi}, \notag\\
I_2&=(1-2\xi)w'(r+t\eta-2t\eta\xi),  \notag \\
I_3&=-\frac{1}{2}(1-\xi)\frac{w(r+t\eta-2t\eta\xi)}{r+t\eta-t\eta\xi}, \\
I_4&=\frac{\xi}{2}\frac{w(r+t\eta-2t\eta\xi)}{r-\xi t\eta}, \notag \\
I_5&=w(r+t\eta-2t\eta\xi)T'_{m-1}(\Theta(r,t,\eta,\xi))\frac{\p}{\p t}\Theta(r,t,\eta,\xi). \notag
\end{align}

{\bf  Estimates for terms involving ${\bf I_1, \cdots, I_4}$:}  For
$\d 0\le \xi \le \frac{1}{2}$ and and $0\le \eta \le 1$, we have:
\begin{align}
\label{est_0^1/2}
\eta(I_1+I_2+I_3+I_4) T'_{m-1}(\Theta)
 & \ge \eta\left(\frac{m-1}{2}+\frac{(3-4m)\xi}{4}\right)
\frac{w(r+t\eta-2t\eta\xi)}{r+t\eta-2t\eta\xi}\notag \\
&-|w'(r+t\eta-2t\eta\xi)| \notag \\
&\ge -\frac{1}{8}\frac{w(r+t\eta-2t\eta\xi)}{r+t\eta-2t\eta\xi}-|w'(r+t\eta-2t\eta\xi)|,
\end{align}
where we have used Lemma \ref{lem_4.1} and Lemma \ref{lem_Xi} with $\lambda=r+t\eta-2t\eta\xi$ and $t=t\eta$ for $0\le \eta,\xi\le1$ and $(r,t)\in \Sigma_1$. 

Similarly, for $\d \frac{1}{2}\le \xi \le 1$, $0\le \eta \le 1$ and $(r,t)\in \Sigma_1$, then Lemma \ref{lem_4.1} and Lemma \ref{lem_Xi}, yield:\begin{align} 
\label{est_1/2^1} 
\eta(I_1+I_2+I_3+I_4) T'_{m-1}(\Theta)
& \ge -\left(m+\frac{1}{8}\right)\frac{w(r+t\eta-2t\eta\xi)}{r+\eta-2t\eta\xi}\notag \\
& -|w'(r+t\eta-2\eta\xi)|. 
\end{align} 

\medskip

{\bf  Estimates for the term involving ${\bf I_5}$:}

In order to obtain the proper estimate for this term, we first aim  to prove the following property:  

For $0\le\xi,\eta \le 1$, $(r,t)\in\Sigma_1$
and  $m\geq 2$, we have 
\begin{align} 
\label{variable_est} 
-\frac{5\zeta_m}{3(r+t\eta-2t\eta\xi)}\le \frac{\p}{\p t}\Theta(r,t,\eta,\xi)\le 0.
\end{align} 

Indeed,  direct computation shows
\begin{align} \label{4.16}
\frac{\p}{\p t}\Theta(r,t,\eta,\xi)
=\frac{\eta N(r,t,\eta,\xi)}{2r(r+t\eta-2t\eta\xi)^2}, 
\end{align} 
where 
\begin{align*}  
N(r,t,\eta,\xi)
&=\left\{2(r+t\eta-2t\eta\xi)(1-2\xi)-2t\eta\right\}(r+t\eta-2t\eta\xi) \\ 
&-\left\{(r+t\eta-2t\eta\xi)^2+r^2-t^2\eta^2\right\}(1-2\xi). 
\end{align*} 
However, a straightforward computation yields
\begin{align} \label{fact_N} 
N(r,t,\eta,\xi)
=&-8t^2\eta^2\xi^3+(12t^2\eta^2+8rt\eta)\xi^2-(8rt\eta+4t^2\eta^2)\xi\notag\\ 
=&-4t\eta\xi(\xi-1)(2t\eta\xi-(2r+t\eta)). 
\end{align} 
Since  $\d \frac{2r+t\eta}{2t\eta}>1$, for $(r,t)\in\Sigma_1$, then 
for each fixed $\eta$ it follows from (\ref{4.16}) and (\ref{fact_N}) that $\frac{\p}{\p t}\Theta(r,t,\eta,\xi)\le 0$.

In order to prove the lower bound for (\ref{variable_est}), we  compute  the minimum  value of $N(r,t,\eta,\xi)$
as a function of $\xi \in [0,1]$; but  for fixed $\eta$. Indeed,

\[ 
\frac{\p}{\p \xi}N(r,t,\eta,\xi)=-24t^2\eta^2\xi^2+4t\eta(6t\eta+4r)\xi-4t\eta(2r+t\eta) =0
\] 
if and only if,
\[ 
\xi=\frac{(3t\eta+2r)\pm \sqrt{3t^2\eta^2+4r^2}}{6t\eta}.
\] 
Put
\[
\xi_{+}=\frac{(3t\eta+2r)+\sqrt{3t^2\eta^2+4r^2}}{6t\eta},\quad 
\xi_{-}=\frac{(3t\eta+2r)-\sqrt{3t^2\eta^2+4r^2}}{6t\eta}.
\]
Obviously, we have $\xi_{+} >1$ and $0<\xi_{-}<1$, for all $(r,t)\in\Sigma_1$ 
and fixed $\eta$. Therefore,  
\[ 
\begin{array}{lll}
N(r,t,\eta,\xi_{-})
&\d=\frac{1}{27t\eta}(3t\eta+2r- \sqrt{3t^2\eta^2+4r^2})(\sqrt{3t^2\eta^2+4r^2}+3t\eta-2r)\\ 
&\d \quad\times (-4r- \sqrt{3t^2\eta^2+4r^2}). 
\end{array} 
\] 
\par 
Here, it is important to note that the following inequalities hold:
\[ 
\frac{3t\eta+2r- \sqrt{3t^2\eta^2+4r^2}}{t\eta}\le3+\frac{2r}{t\eta}-\frac{2r}{t\eta}=3, 
\] 
\[ 
\sqrt{3t^2\eta^2+4r^2}+3t\eta-2r \le \sqrt{4t^2\eta^2+8rt\eta+4r^2}+3t\eta-2r 
\le 5t\eta, 
\] 
\[ 
-4r- \sqrt{3\t^2\eta^2+4r^2}\ge -4r-2(r+t\eta)\ge-6(r+t\eta). 
\] 
for $t,\eta \ge0$. 
Thus,
\[ 
\frac{\p}{\p t}\Theta(r,t,\eta,\xi)\ge -\frac{5}{3r} 
\cdot\frac{t\eta(r+t\eta)}{(r+t\eta-2t\eta\xi)^2}. 
\] 
for all $0\le \xi,\eta \le 1$ and $(r,t)\in\Sigma_1$.
Finally, we note that 
\[ 
r+t\eta \le 2r\quad 
\mbox{and}\quad r+t\eta-2t\eta\xi\ge r-t\eta\ge r-t \ge \frac{2}{\zeta_m}t 
\] 
holds for $0\le \xi,\eta \le 1$ and $(r,t)\in\Sigma_1$. 
Hence,  the lower bound of (\ref{variable_est}) follows. 
\par\noindent
By combining the estimates, (\ref{est_0^1/2}), (\ref{est_1/2^1}), 
(\ref{variable_est}) and (\ref{T_1}), then   (\ref{diff_kernel}) follows, completing the proof of proposition \ref{prop_4.2}.
\hfill$\Box$

By using proposition \ref{prop_4.2}, (\ref{T_1}) and (\ref{asm_even}), then (\ref{u_0_even}) 
implies
\begin{align}
\label{u^0_est_1}
\begin{array}{lll} 
u^0(r,t)
&\d\ge \frac{t}{\pi r^{m-1}}\int_0^1\frac{\eta d\eta}{\sqrt{1-\eta^2}}
\int_{0}^{1}\left\{\frac{f(r+t\eta-2t\eta\xi)}{2t}
-E_m\frac{f(r+t\eta-2t\eta\xi)}{r+t\eta-2t\eta\xi}\right.\\
&\d\left.\quad-|f'(r+t\eta-2t\eta\xi)|+\frac{g(r+t\eta-2t\eta\xi)}{2}\right\}
\frac{K(r,t,\eta,\xi)d\xi}{\sqrt{\xi}\sqrt{1-\xi}}
\end{array} 
\end{align} 
in $\Sigma_1$.
Here we note that 
\[
t\le r-t\eta\le r+(1-2\xi)t\eta
\]
for $(r,t)\in \Sigma_1$. Since $\zeta_m\le 1$ for all $m=2,3,4,\cdots$, then it follows that  
$r\ge2t$. Thus, by returning to the original variables (\ref{change_vari}), then (\ref{u^0_est_pre_even}) follows. 
The proof of the proposition \ref{prop_frame_high} in $n=2m,\ m=2,3,4,\cdots$ is now complete. 
\hfill$\Box$

\smallskip
{\bf Completion of the  Proof of Lemma \ref{lem_frame_high}.} 
Finally, we shall derive the first term in (\ref{frame_high}). 
It follows from (\ref{asm_even}) and (\ref{u^0_est_pre_even}) that 
\begin{align}
\label{u^0_est_2}
\begin{array}{lll} 
u^0(r,t)
&\d\ge \frac{t}{\pi r^{m-1}}\int_0^1\frac{\eta d\eta}{\sqrt{1-\eta^2}}
\int_{0}^{1}\frac{K(r,t,\eta,\xi)}{(1+r+t\eta-2t\eta\xi)^{\kappa+1}}
\frac{d\xi}{\sqrt{\xi}\sqrt{1-\xi}}\\
&\d\ge \frac{C_3t}{\pi\sqrt{2}(1+r+t)^{\kappa+1}}
\int_0^1\frac{\eta d\eta}{\sqrt{1-\eta^2}}
\int_{0}^{1/2}\frac{d\xi}{\sqrt{\xi}\sqrt{1-\xi}}\\
&\d\ge \frac{C_3t}{\pi\sqrt{2}(1+r+t)^{\kappa+1}}
\end{array} 
\end{align} 
in $\Sigma_1$. 
\end{proof}
%%%%%%%%%%%%%%%%%%%%%%%%%%%%%%%%%%%%%%%%%%%%%%%%%%%%%%%%%%%%%%%%%%%%%%%%%%%%%%%%% 
%%%%%%%%%%%%%%%%%%%%%%%%%%%%%%%%%%%%%%%%%%%%%%%%%%%%%%%%%%%%%%%%%%%%%%%%%%%%%%%%% 
\section{Low dimensions: $n=2,3$}\label{sec5}
{\bf Proof of proposition \ref{prop_frame_low}.} 
Let $u^0$ be the solution of  
(\ref{u^0_low}). Then, $u^0$ is given by:
\begin{align}
\label{u^0_low_rep}
u^0(x,t)=\p_t R(f|x,t)+R(g|x,t),
\end{align}
where $R$ is as defined in (\ref{riemman}).
\par 
First, we consider the case $n=3$. By using (\ref{asm_low}), it follows from (\ref{u^0_low_rep}) that
\begin{align}
\label{u^0_est_3dim}
\begin{array}{ll}
u^0(x,t)&\d=\frac{1}{4\pi}\int_{|\omega|=1}\left\{
f+t\omega\cdot\nabla f+tg)\right\}(x+t\omega)dS_{\omega}\\
&\d\ge \frac{t}{4\pi}\int_{|\omega|=1}\left\{
\frac{f}{t}-|\nabla f|+g)\right\}(x+t\omega)dS_{\omega}
\end{array}
\end{align}
for all $(x,t) \in \Sigma_2$. Then,  (\ref{u^0_est_pre_3}) follows  by noting $t\le 1+|x+t\omega|$, for all
$(x,t)\in \Sigma_2$. Furthermore, we easily obtain the first term in (\ref{frame_low}) by substituting
 (\ref{asm_low}) into (\ref{u^0_est_pre_3}).
\par
Next, we consider the case of $n=2$. Here, we make the change variables: 
$\rho=t\xi$ in $\rho$-integral of (\ref{riemman}). Thus, 
\[
R(\phi|x,t)=\frac{t}{2\pi}\int_{0}^{1}\frac{\xi d\xi}{\sqrt{1-\xi^2}}
\int_{|\omega|=1}\phi(x+t\xi\omega)dS_{\omega}.
\]
As in (\ref{u^0_est_3dim}), we obtain
\[
\begin{array}{ll}
u^0(x,t)&\d=\frac{1}{2\pi}\int_{0}^{1}\frac{\xi d\xi}{\sqrt{1-\xi^2}}
\int_{|\omega|=1}\{f+t\xi \omega\cdot \nabla f+tg\}(x+t\xi\omega)dS_{\omega}\\
&\d\ge\frac{t}{2\pi}\int_{0}^{1}\frac{\xi d\xi}{\sqrt{1-\xi^2}}
\int_{|\omega|=1}\left\{\frac{f}{t}-|\nabla f|+g\right\}(x+t\xi\omega)dS_{\omega}
\end{array}
\]
in $\Sigma_2$. Since $t\le 1+|x+t\xi\omega|$ for all
$(x,t)\in \Sigma_2$, then  (\ref{u^0_est_pre_2})  follows,  after going back to the original  variables. 
Furthermore, we easily obtain the first term in (\ref{frame_low}) as in the case of $n=3$. 
Therefore, the proof of lemma \ref{lem_frame_low} is complete.
\end{proof}

\section*{Acknowledgement}
We would like to express our sincere gratitude to Professor Masashi Mizuno 
for showing us the way of a substitution in the decay condition (\ref{asm_even}). This work  began while the fourth author was in his second year of the masters program at the 
Graduate School of Systems Information Science, Future University Hakodate, Japan.

%*********************************BIBLIOGRAPHY**************************

%\nocite{*}
\bibliographystyle{amsplain}

\def\Dbar{\leavevmode\lower.6ex\hbox to 0pt{\hskip-.23ex \accent"16\hss}D}
  \def\hckudot#1{\ifmmode\setbox7\hbox{$\accent"14#1$}\else
  \setbox7\hbox{\accent"14#1}\penalty 10000\relax\fi\raise 1\ht7
  \hbox{\raise.2ex\hbox to 1\wd7{\hss.\hss}}\penalty 10000 \hskip-1\wd7\penalty
  10000\box7} \def\cprime{$'$} \def\cprime{$'$} \def\cprime{$'$}
  \def\cprime{$'$}
\providecommand{\bysame}{\leavevmode\hbox to3em{\hrulefill}\thinspace}
\providecommand{\MR}{\relax\ifhmode\unskip\space\fi MR }
\providecommand{\MRhref}[2]{%
  \href{http://www.ams.org/mathscinet-getitem?mr=#1}{#2}
}
\providecommand{\href}[2]{#2}

\end{document}